\documentclass{amsproc}
\usepackage{amssymb}
\usepackage{amsmath}
\usepackage{mathdots}
\usepackage{amsbsy}
\usepackage{amscd}
\usepackage{amsthm}
\usepackage{fancyvrb}
\usepackage{fvextra}

\usepackage[all]{xy}
\usepackage{mathrsfs, graphicx }

\usepackage{footnote}
\usepackage{hyperref}

\usepackage{setspace}

\usepackage{tikz-cd}
\usetikzlibrary{cd}
\usetikzlibrary{decorations.markings}
\tikzset{negated/.style={
    decoration={markings,
      mark= at position 0.5 with {
        \node[transform shape] (tempnode) {$\times$};
      }
    },
    postaction={decorate}
  }
}
\usepackage{array}
\newcolumntype{P}[1]{>{\raggedright\arraybackslash}p{#1}}
\newtheorem{theorem}{Theorem}[section]
\newtheorem{lemma}[theorem]{Lemma}
\newtheorem{corollary}[theorem]{Corollary}

\newtheorem{proposition}[theorem]{Proposition}
\theoremstyle{remark}
\newtheorem{remark}[theorem]{Remark}
\newtheorem{question}[theorem]{Question}
\theoremstyle{definition}
\newtheorem{definition}[theorem]{Definition}
\newtheorem{example}[theorem]{Example}
\numberwithin{enumi}{theorem}
\newcommand{\inv}{^{-1}}

\title[Simultaneous Conjugacy Classes]{Simultaneous Conjugacy Classes as\\Combinatorial Invariants of Finite Groups}
\author{Dilpreet Kaur}
\address{Indian Institute  of Technology Jodhpur, NH 65, Surpura Bypass Rd, Karwar, Rajasthan 342037, India.}
\email{dilpreetkaur@iitj.ac.in}
\author{Sunil Kumar Prajapati$^*$}
\address{Indian Institute of Technology Bhubaneswar, Arugul Campus, Jatni, Khurda-752050, India.}
\email{skprajapati@iitbbs.ac.in}
\author{Amritanshu Prasad}
\address{The Institute of Mathematical Sciences (Homi Bhabha National Institute), CIT campus Taramani, Chennai 600113, India.}
\email{amri@imsc.res.in}
\thanks{$^{\textbf{*}}$ Corresponding author.\\ Dilpreet Kaur was supported by SERB National Postdoctoral Fellowship PDF/2017/000188 of the Department of Science \& Technology, India. Sunil Kumar Prajapati acknowledges the support from project grant SP096, given by Indian Institute of Technology Bhubaneswar and  SERB,  Government  of  India  for  financial  support  through  grant(MTR/2019/000118). Amritanshu Prasad was supported by a Swarnajayanti Fellowship of the Department of Science \& Technology, India}
\subjclass[2010]{20E45, 20D15, 20D60}
\keywords{simultaneous conjugacy classes, finite groups, isoclinic groups}
\begin{document}
\maketitle
\begin{abstract}
Let $G$ be a finite group.
We consider the problem of counting simultaneous conjugacy classes of $n$-tuples and simultaneous conjugacy classes of commuting $n$-tuples in $G$. 
Let $\alpha_{G,n}$ denote the number of simultaneous conjugacy classes of $n$-tuples, and $\beta_{G,n}$ the number of simultaneous conjugacy classes of commuting $n$-tuples in $G$.
The generating functions
$A_G(t) = \sum_{n\geq 0} \alpha_{G,n}t^n,$ 
and
$B_G(t) = \sum_{n\geq 0} \beta_{G,n}t^n$ are rational functions of $t$.
We show that $A_G(t)$ determines and is completely determined by the class equation of $G$.
We show that $\alpha_{G,n}$ grows exponentially with growth factor equal to the cardinality of $G$, whereas $\beta_{G,n}$ grows exponentially with growth factor equal to the maximum cardinality of an abelian subgroup of $G$.
The functions $A_G(t)$ and $B_G(t)$ may be regarded as combinatorial invariants of the finite group $G$.
We study dependencies amongst these invariants and the notion of isoclinism for finite groups.
We prove that the normalized functions $A_G(t/|G|)$ and $B_G(t/|G|)$ are invariants of isoclinism families.
\end{abstract}
\section{Introduction}
\label{sec:introduction}
Any group $G$ acts on its Cartesian power $G^n$ by simultaneous conjugation:
\begin{displaymath}
  g\cdot(x_1,\dotsc,x_n) = (gx_1 g\inv,\dotsc,gx_ng\inv).
\end{displaymath}
Let $G^{(n)}$ denote the subset of $G^n$ consisting of pairwise commuting tuples:
\begin{displaymath}
  G^{(n)} = \{(x_1,\dotsc,x_n)\in G^n \mid [x_i, x_j] = 1 \text{ for all } 1\leq i, j \leq n\}.
\end{displaymath}
Clearly $G^{(n)}$ is also closed under the above $G$-action. 
Even when conjugacy classes in $G$ are well-understood, it may not be easy to understand the $G$-orbits in $G^n$ and $G^{(n)}$.
For example, when $G$ is the matrix group $GL_m(F)$ for some field $F$, the determination of orbits in $G^2$ is the matrix pair problem, the quintessential wild problem in algebra.
The determination of orbits in $GL_m(F)^{(n)}$ corresponds to the classification of $m$ dimensional modules for the polynomial algebra $F[x_1,\dotsc,x_n]$ up to isomorphism.
This problem was solved for $m\leq 4$ when $F$ is a finite field by Sharma \cite{MR3485060}.

In this article, we restrict ourselves to finite groups and focus on the combinatorial problems of counting $G$-orbits in $G^n$ and $G^{(n)}$.
Let $\alpha_{G,n}$ denote the number of $G$-orbits in $G^n$, and $\beta_{G,n}$ the number of $G$-orbits in $G^{(n)}$.
Consider the generating functions:
\begin{displaymath}
  A_G(t) = \sum_{n=0}^\infty \alpha_{G,n}t^n,
\end{displaymath}
and
\begin{displaymath}
  B_G(t) = \sum_{n=0}^\infty \beta_{G,n}t^n.
\end{displaymath}
Note that $G^0$ and $G^{(0)}$ are the trivial group, and so $\alpha_{G,0} = \beta_{G,0}=1$. Further, note that $\alpha_{G,n} \geq \beta_{G,n}$ with $\alpha_{G,1} = \beta_{G,1}$, both equal to the number of conjugacy classes of $G$. 

We derive expressions (\ref{eq:3}) and (\ref{eq:5}) for $A_G(t)$ and $B_G(t)$, which show that they are rational functions of $t$.
The rational functions $A_G(t)$ and $B_G(t)$ may be regarded as combinatorial invariants of $G$.
The rational functions for $A_G(t)$ and $B_G(t)$ have simple poles that lie on the positive real axis.
By locating these poles, we show that $\alpha_{G,n}$ grows like a geometric series in $n$ with common ratio $|G|$, the cardinality of $G$, whereas $\beta_{G,n}$ grows with common ratio equal to the maximal cardinality of an abelian subgroup of $G$ (Theorem~\ref{theorem:asymptotic}).

We say that finite groups $G_1$ and $G_2$ are $A$-equivalent (resp., $B$-equivalent) if $A_{G_1}(t)=A_{G_2}(t)$ (resp., $B_{G_1}(t)=B_{G_2}(t)$).
$A$-equivalence is easy to characterize: two finite groups are $A$-equivalent if and only if they have the same class equation (Theorem~\ref{theorem:class-eq-A}).
We do not have a similar group-theoretic characterization of $B$-equivalence in general. However in case of a $p$-group $G$, coefficient of $t^n$ in $B_G(t)$ has an interesting meaning in algebraic topology: it is the Morava $K$-theory Euler characteristic the classifying space of $G$ \cite[Theorem B (Part 1)]{HKR}.
We have computed both $A_G(t)$ and $B_G(t)$ for all $p$-groups of rank up to $5$ in \cite{KPP}.  

The normalized functions $A_G(t/|G|)$ and $B_G(t/|G|)$ are invariants of isoclinism families (Theorem~\ref{corollary:family-invariants}).
Therefore isoclinic groups of the same order are both $A$-equivalent and $B$-equivalent.

By searching the small groups library of GAP \cite{GAP4}, all other possible dependencies amongst these three equivalence relations are ruled out and the counterexamples of smallest order are provided (Section~\ref{sec:counterexamples}).
For convenience, the group of order $n$ with GAP small groups identifier $(n,r)$ will be denoted $G(n,r)$.
The groups $G(54,6)$ and $G(54,8)$ have centres of different order, hence are not $A$-equivalent, but they are $B$-equivalent.
The groups $G(128,2022)$ and $G(128,1758)$ are $A$-equivalent, but not $B$-equivalent.
The group $G(18,1)$ (dihedral of order $18$) and $G(18,4)$ are $A$-equivalent and $B$-equivalent, but not isoclinic.

The relationship between these equivalences is summarized in Figure~\ref{fig:dependencies}.
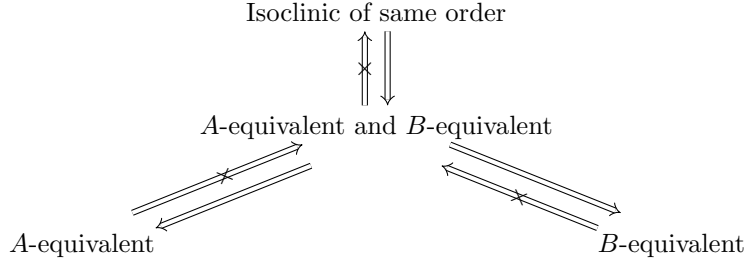
\begin{figure}
  \label{fig:dependencies}
  \begin{center}
    \begin{tikzcd}[arrows=Rightarrow, column sep=0.3cm, row sep=1cm, every arrow/.append style={shift left=1ex}]
      {}  
    & \text{Isoclinic of same order} 
    \arrow{d} 
    & {}
    \\
    {} 
    & \text{$A$-equivalent and $B$-equivalent}
    \arrow[negated]{u} 
    \arrow[shorten <= 6pt,shorten >= 6pt]{dl}
    \arrow[shorten <= 6pt]{dr}
    & {}
    \\
    \text{$A$-equivalent} 
    \arrow[negated, shorten >= 6pt]{ur}
    & {} 
    & \text{$B$-equivalent}
    \arrow[negated, shorten <= 6pt,shorten >= 6pt]{ul}
  \end{tikzcd}
\end{center}
\caption{Dependencies between equivalence relations}
\end{figure}
In Section ~\ref{GAP: code}, we provide the GAP code of the 
functions $A_G(t)$ and $B_G(t)$ for the convenience of the reader. 
\section{Rationality of $A_G(t)$ and $B_G(t)$}
In this section, we present algorithms to compute the formal power series $A_G(t)$ and $B_G(t)$, from which it follows that they are rational functions of $t$.

For each $g\in G$, let $Z_G(g)$ denote its centralizer.
The function $A_G(t)$ can be computed by a simple application of Orbit counting lemma to the action of $G$ on $G^n$:
\begin{equation}
  \label{eq:2}
  \alpha_{G,n} = \frac 1{|G|} \sum_{g\in G} |Z_G(g)|^n.
\end{equation}
Therefore
\begin{align}
  \nonumber
  A_G(t) & = \sum_{n=0}^\infty \alpha_{G,n} t^n\\
  \nonumber
  & = \sum_{n=0}^\infty \frac 1{|G|} \sum_{g\in G}|Z_G(g)|^n t^n\\
  \label{eq:3}
  & = \frac 1{|G|}\sum_{g\in G} \frac 1{1-|Z_G(g)|t},
\end{align}
a rational function of $t$.

In the following theorem, we show that information contained in the generating function $A_G(t)$ is the same as the information contained in the class equation of
$G$.
\begin{theorem}
  \label{theorem:class-eq-A}
  Let $G$ and $H$ be finite groups.
  Then $A_G(t) = A_H(t)$ if and only if $G$ and $H$ have the same class equation.
\end{theorem}
\begin{proof}
  Let $z_m$ denote the number of elements of $G$ having centralizer of cardinality $m$ (and therefore conjugacy class of cardinality $|G|/m$).
  Note that $z_m=0$ if $m>|G|$.
  Clearly, the sequence $z_1,z_2,\dotsc, z_N$, where $N = |G|$ completely determines, and is completely determined by the class equation of $G$.
  The expression \eqref{eq:2} for $A_G(t)$ can be rewritten as:
  \begin{equation}
    \label{eq:4}
    \alpha_{G,n} = \frac 1{|G|}\sum_{m = 1}^\infty z_m m^n.
  \end{equation}
  Thus the coefficients of $A_G(t)$ are completely determined by the class equation of $G$.

  Conversely, let $N = |G|$.
  The identity \eqref{eq:4} can be viewed as a matrix identity:
  \begin{equation}
    \label{eq:7}
    \begin{pmatrix}
      \alpha_{G,1}\\
      \alpha_{G,2}\\
      \vdots\\
      \alpha_{G,N}
    \end{pmatrix}
     =
     \frac 1{|G|}
    \begin{pmatrix}
      1 & 2 & \cdots & N\\
      1^2 & 2^2 & \cdots & N^2\\
      \vdots & \vdots & \ddots & \vdots\\
      1^N & 2^N & \cdots & N^N
    \end{pmatrix}
    \begin{pmatrix}
      z_1\\
      z_2\\
      \vdots\\
      z_N
    \end{pmatrix}.
  \end{equation}
  The non-singularity of the Vandermonde matrix in this identity implies that the sequence $z_1,\dotsc, z_N$ (and hence the class equation of $G$) is completely determined by the first $N$ coefficients of $A_G(t)$.
\end{proof}

The expression \eqref{eq:4} suggests an alternative form for Eq.~\eqref{eq:3}:
\begin{equation}
  \label{eq:6}
  A_G(t) = \frac 1{|G|}\sum_{m = 1}^\infty \frac{z_m}{1-mt},
\end{equation}
which gives the partial fraction decomposition of $A_G(t)$.

Orbit counting lemma cannot be used to compute $B_G(t)$ in this way, because the entries of a tuple in $G^{(n)}$ are not independently chosen (they must commute with each other).
Instead, a recursive process is used.

Let $g\in G$.
The map $(g, g_1,\dotsc, g_{n-1})\to (g_1,\dotsc,g_{n-1})$ induces a bijection from the set of $G$-orbits in $G^{(n)}$ which contain an element whose first coordinate is $g$ onto the set of $Z_G(g)$-orbits in $Z_G(g)^{(n-1)}$.
Let $c_H$ denote the number of conjugacy classes of $G$ whose centralizer is isomorphic to a subgroup $H$ of $G$. 
Then the bijection that we have just defined gives rise to the identity
\begin{equation*}
  \beta_{G,n} = \sum_H c_H \beta_{H,n-1},
\end{equation*}
for each $n\geq 1$, the sum on the right hand side being over isomorphism classes of subgroups of $G$.
Consequently,
\begin{align*}
  B_G(t) & = 1 + \sum_{n=1}^\infty \beta_{G,n}t^n\\
  & = 1 + \sum_{n=1}^\infty \sum_H c_H \beta_{H,n-1}t^n\\
  & = 1 + \sum_H c_H t B_H(t).
\end{align*}
Bringing the term with $H = G$ from the right hand side to the left gives
\begin{equation}
  \label{eq:5}
  (1 - c_Gt)B_G(t) = 1 + \sum_{|H|<|G|} c_H t B_H(t).
\end{equation}
Note that $c_G$ is the cardinality of the centre of $G$.
The above identity establishes the rationality of $B_G(t)$ by induction on $|G|$ (the base cases are groups of prime order, which are abelian, and therefore $B_G(t)$ is obviously rational for them).

\section{Asymptotic Formulas}
\label{sec:expon-growth-rates}
If a power series $\sum_{n=0}^\infty a_n t^n$ represents a proper rational function given as a partial fraction:
\begin{equation}
  \label{eq:par-fra}
  \sum_{n=0}^\infty a_n t^n = \sum_{i=1}^N \frac{c_i}{1-r_it},
\end{equation}
where $c_i$ and $r_i$ are real numbers for $i=1,\dotsc,N$, and $r_1>r_2>\dotsb>r_N$, then
\begin{displaymath}
  a_n = \sum_{i=1}^N c_i r_i^n,
\end{displaymath}
from which we get:
\begin{equation}
  \label{eq:asymptotic}
  a_n \sim c_1r_1^n.
\end{equation}
Here, we write $a_n\sim b_n$ if $\lim_{n\to \infty} a_n/b_n=1$.

\begin{theorem}
  \label{theorem:asymptotic}
  For any finite group $G$, we have:
  \begin{align*}
    \alpha_{G,n}&\sim |Z(G)||G|^{n-1},\\
    \beta_{G,n}&\sim Ca^n,
  \end{align*}
  where $a$ is the maximal cardinality of an abelian subgroup of $G$, and $C$ is some positive constant.
\end{theorem}
\begin{proof}
  Let $r_1>r_2>\dotsb>r_N$ denote the set of orders of the centralizers $Z_G(g)$ as $g$ varies in $G$.
  Then collecting summands with the same centralizer degree, \eqref{eq:3} can be rewritten in the form \eqref{eq:par-fra}.
  The constant $c_i$ is the probability that an element of $G$, chosen uniformly at random, has centralizer of order $r_i$. 

  Given groups $H$ and $K$, let $c^H_K$ denote the number of conjugacy classes of elements in $H$ whose centralizer in $H$ is isomorphic to $K$.
  In particular, $c^H_H$ is the cardinality of the centre of $H$.
  The sum on the right hand side of the recursive identity (\ref{eq:5}) is trivial if and only if $G$ is abelian.
  Otherwise, each non-zero term on the right hand side corresponds to a subgroup $H$ of $G$ which is the centralizer of a non-central element $g\in G$.
  Thus the centre of $H$ contains the centre of $G$ as well as the non-central elements of $G$, so that $c^H_H>c^G_G$.
  Each path in the recursion tree when (\ref{eq:5}) is iterated leads to a term of the form:
  \begin{displaymath}
   \bigg( \frac 1{1-c^G_Gt}\bigg) \times \bigg( \frac{c^G_{H_1}t}{1-c^{H_1}_{H_1}t}\bigg) \times \dotsb \times \bigg(\frac{c^{H_{k-1}}_{H_{k-1}}t}{1-c^{H_{k-1}}_{H_k}t}\bigg),
  \end{displaymath}
  where $G=H_0,H_1,\dotsc,H_k=H$ is a sequence of groups such that $H_i$ is the centralizer of a non-central element of $H_{i-1}$ for $i=1,\dotsc,k$ and $H_k$ is abelian.
  Thus, when $B_G$ is expanded in a partial fraction of the form (\ref{eq:par-fra}), where $r_1$ is the maximal cardinality of an abelian subgroup $H$ of $G$ for which there exists a sequence $G=H_0,H_1,\dotsc,H_k=H$ and elements $g_i\in H_{i-1}$, with $H_i=Z_{H_{i-1}}(g_i)$ for each $i=1,\dotsc, k$.
  Equivalently, $H$ is an abelian subgroup of maximal cardinality of the form:
  \begin{equation}
    \label{eq:growthB}
    H = Z_G(g_1)\cap Z_G(g_2)\cap \dotsb \cap Z_G(g_k),
  \end{equation}
  where $(g_1,\dotsc,g_k)\in G^{(k)}$.

  Let $H$ be \emph{any} maximal abelian subgroup of $G$.
  Let $k=|H|$ and let $g_1,\dotsc,g_k$ be an enumeration of the elements of $H$.
  We claim that (\ref{eq:growthB}) holds, so every maximal abelian subgroup of $G$ is of this form.
  Indeed, the right hand side of (\ref{eq:growthB}) consists of all the elements of $G$ that commute with every element of $H$.
  By maximality of $H$, this is exactly the set of elements of $H$.
  This proves the assertion about the growth of $\beta_{G,n}$.
\end{proof}

\section{$A_G(t), B_G(t)$ and isoclinism}
\label{sec:a_gt-b_gt-isoclinism}
In 1940, Hall \cite{Hall} introduced isoclinism, an equivalence relation on the class of all groups.
Isoclinism makes precise the idea that two groups have the same commutator function.
This section is devoted to studying the relation between isoclinism of groups and the formal series $A_G(t), B_G(t)$.
\begin{definition}[Isoclinism]
  Two  finite groups $G$ and $H$ are said to be \emph{isoclinic} if
  there exist isomorphisms $\theta : G/Z(G)\longrightarrow H/Z(H)$ and $\phi: G{}'\longrightarrow H{}'$
  such that the following diagram is commutative:
  \begin{equation*}
    \begin{tikzcd}
      G/Z(G)\times G/Z(G) \arrow{d}{a_G} \arrow{r}{\theta \times \theta}
      & H/Z(H)\times H/Z(H) \arrow{d}{a_H} \\
      G{}' \arrow{r}{\phi}
      & H{}',
    \end{tikzcd}
  \end{equation*}
  where $a_G(g_1Z(G), g_2Z(G)) = [g_1,g_2]$, for $g_1,g_2\in G$, and $a_H(h_1Z(H), h_2Z(H)) = [h_1,h_2]$ for $h_1,h_2\in H$.
\end{definition}
The resulting pair $(\theta, \phi)$ is called an \emph{isoclinism} of $G$ onto $H$.
Isoclinism is an equivalence relation on groups strictly weaker than isomorphism.
For instance, a group is isoclinic to the trivial group if and only if it is abelian.
More generally, isoclinic nilpotent groups have the same nilpotency class.
\renewcommand{\theenumi}{\thetheorem.\arabic{enumi}}
\begin{remark}[{\cite[p.~134]{Hall}}]\label{remark_isoclinism}
  Let $(\theta, \phi)$ be an isoclinism from $G$ onto $H$.
  \begin{enumerate}
  \item \label{item:9}
  Let $q_G:G\to G/Z(G)$ and $q_H:H\to H/Z(H)$ denote the quotient maps.
  For each subgroup $K$ of $G$ containing $Z(G)$ let $\bar\theta(K)=q_H^{-1}(\theta(q_G(K)))$.
  Then $\bar\theta$ is a bijection from the set of subgroups of $G$ containing $Z(G)$ onto the set of subgroups of $H$ containing $Z(H)$, and $K$ and $\bar\theta(K)$ are isoclinic.
  For any element $g\in G$, let $h\in H$ be such that $\theta(q_G(g))=q_H(h)$.
  Then $\bar\theta(Z_G(g))=Z_H(h)$.
  In particular, a centralizer in the group $G$ is isoclinic to a centralizer in $H$.
\item The isoclinism  $(\theta, \phi)$ from $G$ onto $H$ gives rise to a bijective correspondence between quotient groups $G/K$ and $H/L$, where $K\subseteq G{}'$ and $L\subseteq H{}'$, and corresponding quotient groups are isoclinic.
\end{enumerate}
\end{remark}
\begin{remark}[{\cite[Lemma 1.3(a)]{JCB}}]\label{remark_subgroup_isoclinism}
 Let $G$ be a group and $H$ be its subgroup such that $HZ(G)=G,$ where $Z(G)$ denotes the center of group $G.$ Then group $G$ is isoclinic to $H.$
\end{remark}

Isoclinism has strong implications for $A_G(t)$ and $B_G(t)$:
\begin{theorem}\label{A_B_isoclinism}
  Let $G$ and $H$ be isoclinic groups.
  Then
  \begin{enumerate}
  \item  $A_G(t)=A_H\left(\frac{|G|}{|H|}t\right).$
  \item  $B_G(t)=B_H\left(\frac{|G|}{|H|}t\right).$
  \end{enumerate}
\end{theorem}
\begin{proof}
  Let $(\theta,\phi)$ be an isoclinism of $G$ onto $H$.
  Then $\theta$ induces a bijection from the elements of $G/Z(G)$ onto elements of $H/Z(H)$ which preserves the orders of centralizers.
  Let $z_m$ denote the number of elements of $G/Z(G)$ with centralizer of order, and let $z^G_m$ denote the number of elements of $G$ with centralizer of order $m$.
  It follows that, for every positive integer $m$, $z^G_{|Z(G)|m}=|Z(G)|z_m$.
  Also, $z^G_m=0$ if $m$ is not a multiple of $|Z(G)|$.
  By (\ref{eq:6}), we get:
  \begin{displaymath}
    A_G(t) = \frac 1{|G|}\sum_m \frac{|Z(G)|z^G_m}{1-|Z(G)|mt} = A_{G/Z(G)}(|Z(G)|t).
  \end{displaymath}
  A similar identity holds for $H$.
  Since $G/Z(G)$ is isomorphic to $H/Z(H)$, $A_{G/Z(G)}=A_{H/Z(H)}$.
  Therefore
  \begin{align*}
    A_G(t) & = A_{G/Z(G)}(|Z(G)|t)\\
           & = A_{H/Z(H)}(|Z(G)|t)\\
           & = A_{H}(|Z(G)||Z(H)|^{-1}t)\\
           & = A_{H}(|G||H|^{-1}t)
  \end{align*}
  and the first identity of the theorem follows.

  For the second identity, assume $|G|\leq |H|$ and proceed by induction on $|G|$.
  If $G$ is an abelian group then $H$, being isoclinic to $G$ is also abelian.
  We have:
  \begin{displaymath}
    B_G(t)=\frac{1}{1-|G|t}=\frac{1}{1-|H|(|G||H|^{-1}t)}=B_H(|G||H|^{-1}t).   
  \end{displaymath}
  In particular, the identity holds for all groups of prime order.

  Suppose the identity holds for groups of order less than $|G|$.
  When $B_G(t)$ is computed using (\ref{eq:5}), the sum on the right hand side consists of subgroups $K$ of $G$ containing $Z(G)$.
  For each such $K$, we have:
  \begin{displaymath}
    B_K(t)=B_{\bar\theta(K)}(|Z(G)||Z(H)|^{-1}t)
  \end{displaymath}
  (by induction hypothesis and Remark~\ref{item:9}).
  Also,
  \begin{displaymath}
    c^G_K = c^{G/Z(G)}_{K/Z(G)}|Z(G)| = c^{H/Z(H)}_{\theta(K/Z(G))}|Z(G)| = c^H_{\bar\theta(K)}|Z(G)||Z(H)|^{-1}.
  \end{displaymath}
  And finally,
  \begin{displaymath}
    |G|/|H|=|Z(G)|/|Z(H)|=|K|/|\bar\theta(K)|.
  \end{displaymath}
  Applying the above observations to (\ref{eq:5}) gives
  \begin{align*}
    (1-|Z(G)|t)B_G(t) & = 1 + t\sum_{K\subset G} c^G_KB_K(t)\\
                & = 1 + t|Z(G)||Z(H)|^{-1}\sum_{K\subset G} c^H_{\bar\theta(K)}B_{\bar\theta(K)}(|Z(G)||Z(H)|^{-1}t)\\
                & = (1-|Z(G)|t)B_H(|Z(G)||Z(H)|^{-1}t),
  \end{align*}
  which is equivalent to the second identity.
\end{proof}
\begin{corollary}
  \label{A_B_for_isoclinic_same_order}
  Isoclinic groups of the same order are $A$-equivalent and $B$-equivalent.
\end{corollary}
\begin{proof}
  This is an immediate consequence of Theorem~\ref{A_B_isoclinism}.
\end{proof}
\begin{remark}
  Corollary~\ref{A_B_for_isoclinic_same_order} and Theorem~\ref{theorem:class-eq-A} imply the well-known fact that isoclinic groups of the same order have the same class equation.
\end{remark}
\begin{corollary}
  Let $G$ be a group and $H$ be an abelian group. Then
  $A_{G\times H}(t)=A_{G}(|H|t)$ and $B_{G\times H}(t)=B_{G}(|H|t)$.
\end{corollary}
\begin{proof}
  This follows from the fact that $G$ is isoclinic to $G\times H$ when $H$ is abelian, and Theorem~\ref{A_B_isoclinism}.
\end{proof}
\begin{definition}
  [Family Invariant, \cite{Hall}]
  A quantity depending on a variable group is called a \emph{family invariant} if it is the same for any two isoclinic groups.
\end{definition}
Another immediate corollary of Theorem~\ref{A_B_isoclinism} is:
\begin{theorem}
  \label{corollary:family-invariants}
  The rational functions $A_G(t/|G|)$ and $B_G(t/|G|)$ are family invariants of the group $G$.
\end{theorem}
\section{Counterexamples found by GAP}
\label{sec:counterexamples}
The converse of Theorem \ref{A_B_isoclinism} is not true.
Moreover neither $A$-equivalence nor $B$-equivalence implies the other.
The assertions in the following examples can be verified using GAP. 
The GAP code for computing $A_G(t)$ and $B_G(t)$ is given in the Appendix.
\begin{example}[$A$-equivalent and $B$-equivalent, but not isoclinic]
Consider the dihedral group $G(18,1)$ with presentation
\begin{equation*}
  \langle \alpha, \beta \mid \alpha^9=\beta^2=\beta^{-1}\alpha\beta\alpha= 1\rangle
\end{equation*}
and the generalized dihedral group $G(18,4)$ with presentation
\begin{equation*}
  \langle \alpha_1,\alpha_2, \beta \mid \alpha_1^3=\alpha_2^3=\beta^2=[\alpha_1,\alpha_2]=\beta^{-1}\alpha_1\beta\alpha_1=\beta^{-1}\alpha_2\beta\alpha_2=1\rangle.
\end{equation*}
Then
\begin{gather*}
  A_{G(18,1)}=A_{G(18,4)} = \frac{-98t^2 + 23t - 1}{324t^3 - 216t^2 + 29t - 1},\\
  B_{G(18,1)}=B_{G(18,4)} = \frac{-t^2 + 6t - 1}{18t^3 - 29t^2 + 12t - 1}.
\end{gather*}
But these groups are not isoclinic, since they both have trivial centre, but are not isomorphic to each other--$G(18,1)$ has exponent $18$, while $G(18,4)$ has exponent $6$.
\end{example}
\begin{example}[$B$-equivalent, but not $A$-equivalent]
Consider the group $G(54,6)$ with polycyclic presentation:
\begin{multline*}
  \langle \alpha_1, \alpha_2, \alpha_3, \beta \mid \alpha_1^2=\alpha_2^3=\alpha_3^9=\beta^3=1,~ \alpha_1\alpha_3\alpha_1^{-1}=\beta\alpha_3^2,~ \alpha_1\beta\alpha_1^{-1}=\beta^{-1},\\  \alpha_2\alpha_3\alpha_2^{-1}=\beta^{-1}\alpha_3 \rangle
\end{multline*}
and group $G(54,8)$ with polycyclic presentation:
\begin{multline*}
  \langle \alpha_1, \alpha_2, \alpha_3, \beta \mid \alpha_1^2=\alpha_2^3=\alpha_3^3=\beta^3=1,~ \alpha_1\alpha_2\alpha_1^{-1}=\alpha_2^{-1}, \\\alpha_1\alpha_3\alpha_1^{-1}=\alpha_3^{-1},~ \alpha_2\alpha_3\alpha_2^{-1}=\beta^{-1}\alpha_3 \rangle.
\end{multline*}
Then
\begin{align*}
  A_{G(54,6)}(t)&=\frac{-181764t^4+44604t^3-3477t^2+104t-1}{1417176t^5-551124t^4+75330t^3-4455t^2+114t-1
},\\
  A_{G(54,8)}(t)&=\frac{-390t^2+59t-1}{2916t^3-864t^2+69t-1
},
\end{align*}
whereas
\begin{align*}
  B_{G(54,6)}=B_{G(54,8)}= \frac{-3t^2+8t-1}{162t^3-99t^2+18t-1
}.\\
\end{align*}
\end{example}
\begin{example}[$A$-equivalent, but not $B$-equivalent]
Consider the group \newline $G(128, 2022)$ with polycyclic presentation:
\begin{multline*}
  \langle \alpha_1, \alpha_2, \alpha_3, \alpha_4, \alpha_5, \alpha_6, \beta \mid \alpha_1^2=\alpha_2^2=\alpha_3^2=\alpha_4^2=\alpha_5^4=\alpha_6^4=\beta^2=1,\\ \alpha_1\alpha_2\alpha_1^{-1}=\beta\alpha_5\alpha_2,~ \alpha_1\alpha_3\alpha_1^{-1}=\beta\alpha_6\alpha_3,~ \alpha_1\alpha_5\alpha_1^{-1}=\beta\alpha_5,~\alpha_1\alpha_6\alpha_1^{-1}=\beta\alpha_6,~\alpha_2\alpha_3\alpha_2^{-1} \\=\beta\alpha_3,~\alpha_2\alpha_4\alpha_2^{-1}=\beta\alpha_5\alpha_4,~\alpha_2\alpha_5\alpha_2^{-1}=\beta\alpha_5,~
\alpha_3\alpha_6\alpha_3^{-1}=\beta\alpha_6,~
\alpha_4\alpha_5\alpha_4^{-1}=\beta\alpha_5\rangle
\end{multline*}
and $G(128, 1758)$ with polycyclic presentation:
\begin{multline*}
  \langle \alpha_1, \alpha_2, \alpha_3, \alpha_4, \alpha_5, \alpha_6, \beta \mid \alpha_1^2=\alpha_2^2=\alpha_3^2=\alpha_4^2=\alpha_5^2=\alpha_6^2=\beta^2=1,\\ \alpha_1\alpha_2\alpha_1^{-1}=\alpha_5\alpha_2,~ \alpha_1\alpha_3\alpha_1^{-1}=\alpha_6\alpha_3,~ \alpha_1\alpha_4\alpha_1^{-1}=\beta\alpha_4,~\alpha_2\alpha_6\alpha_2^{-1}\\=\beta\alpha_6,~\alpha_3\alpha_5\alpha_3^{-1}=\beta\alpha_5 \rangle.
\end{multline*}
Then
\begin{eqnarray*}
  A_{G(128,1758)}=A_{G(128,2022)}=\frac{-187904t^3+12176t^2-211t+1}{4194304t^4-491520t^3+17920t^2-240t+1
},
\end{eqnarray*}
whereas
\begin{eqnarray*}
B_{G(128,1758)} & = \frac{-4t^2-t-1}{512t^3-224t^2+28t-1}
 \end{eqnarray*}
 and
\begin{eqnarray*}
  B_{G(128,2022)} & = \frac{-256t^4-2448t^3-34t^2+33t-1
}{32768t^5-31744t^4+9920t^3-1240t^2+62t-1
}. 
\end{eqnarray*}
\end{example}
\section{AC-Groups}
\label{sec:ac-gps}
A finite group $G$ is called an $AC$-group if the centralizer of every non-central element of $G$ is abelian.
In this section we prove that two AC-groups of the same order are $A$-equivalent if and only if they are $B$-equivalent.
Recall the following characterization of AC-groups:
\begin{lemma}\textnormal{\cite[Lemma 3.2]{Rocke}}\label{acgroup}
  A finite group $G$ is an AC-group if and only if, for all $x,y \in G\setminus Z(G)$ such that $[x,y]=1$, $Z_G(x)=Z_G(y)$.
\end{lemma}
\begin{remark} If $G$ is an $AC$ group, then
$\{Z_G(x)~|~ x\in G\setminus Z(G)\}$ is its set of maximal abelian subgroups.
\end{remark}
\begin{theorem}\label{thm:AC_AG=BG} Let $G$ and $H$ be two AC-groups of same order.
Then $A_G(t)=A_H(t)$ if and only if $B_G(t)=B_H(t)$. 
\end{theorem}
\begin{proof}
  Let $n_1,\ldots, n_r$ (with $n_1=|G|> n_2>\cdots >n_r$) be the distinct cardinalities of the centralizers of elements of $G$ and 
  $m_1,\ldots, m_l$ (with $m_1=|H|> m_2>\cdots >m_l$) be the distinct cardinalities of the centralizers of elements of $H$.
  For each positive integer $n$, let $C_G(n)$ (resp. $C_H(n)$) denote the number of conjugacy classes of $G$ (resp. $H$) whose centralizer has cardinality $n$.
  Now, for any abelian group $K$, $B_{K}(t)=(1-|K|t)^{-1}$.
  Using this  and equation (\ref{eq:5}), we have
\begin{align}\label{eq:AC_equal_B_G(t)}
B_G(t) & = \frac{1}{1-|Z(G)|t} + \frac{1}{1-|Z(G)|t}\sum_{i=2}^r \frac{C_G(n_i)t}{1-n_it} \nonumber \\
 & = \frac{-1}{|Z(G)|t-1} + \sum_{i=2}^r \bigg(\frac{C_G(n_i)/(n_i-|Z(G)|)}{|Z(G)|t-1} 
          +\frac{-C_G(n_i)/(n_i-|Z(G)|)}{n_it-1}\bigg)\nonumber \\
                     & = \frac{-1+\sum_{i=2}^r C_G(n_i)/(n_i-|Z(G)|)}{|Z(G)|t-1} + \sum_{i=2}^r \frac{-C_G(n_i)/(n_i-|Z(G)|}{n_it-1},
\end{align}
the partial fractions expansion of $B_G(t)$.
Similarly, the partial fractions expansion of $B_H(t)$ is
 \begin{align}\label{eq1:AC_equal_B_H(t)}
B_H(t) & = \frac{-1+\sum_{i=2}^l C_H(m_i)/(m_i-|Z(H)|)}{|Z(H)|t-1} + \sum_{i=2}^l \frac{-C_H(m_i)/(m_i-|Z(H)|)}{m_it-1}.
\end{align} 
By the uniqueness of partial fractions decompositions, $B_G(t)=B_H(t)$ if and only if $r=l$, $n_i=m_i$ and $C_G(n_i)=C_H(m_i)$ for $i=1,\dotsc,r$. When $|G|=|H|$, this is equivalent to $G$ and $H$ having the same class equation, and hence (by Theorem~\ref{theorem:class-eq-A}, equivalent to $A_G(t)=A_H(t)$.
\end{proof}

We end this section by mentioning the divisibility relation between $|G|$ and the leading coefficient of the denomination function of rational function $B_G(t)$.

\begin{proposition}\label{prop:divisibility_B_G(t)} Let $G$ be a finite AC-group and let $B_G(t)=\frac{p_G(t)}{q_G(t)}$, where $p_G(t), q_G(t)\in \mathbb{Z}[x]$. Then the leading coefficient of $q_G(t)$ is divisible by $|G|$. 
\end{proposition}
\begin{proof}
  Suppose $n_1,\ldots, n_r$ (with $n_1=|G|> n_2>\cdots >n_r$ and $n_r > |Z(G)|$) are the cardinalities of the centralizers of elements of $G$.
  By (\ref{eq:AC_equal_B_G(t)}), if we write $B_G(t)=p(t)/q(t)$ in lowest terms, then the leading coefficient of $q(t)$ is $|Z(G)|\prod_{i=2}^r n_i$.
  Thus is suffices to prove that $|Z(G)|\prod_{i=2}^r n_i$ is divisible by $|G|$.
 To see this, consider the class equation  
 \begin{displaymath}
   |G|=|Z(G)|+ \sum_{i=2}^rC_G(n_i)\frac{|G|}{n_i}
\end{displaymath}
of $G$. Multiply both sides of the above equation by 
$\prod_{i=2}^r n_i$ to get 
\begin{eqnarray*} 
|G|\prod_{i=2}^r n_i&=& |Z(G)|\prod_{i=2}^r n_i + \sum_{i=2}^r k_iC_G(n_i)|G|~~~~~~~~~~~~~\hspace{.3cm}(\textnormal{for some}~ k_i\in \mathbb{N}).
\end{eqnarray*}
This shows that $|G|$ divides $|Z(G)|\prod_{i=2}^r n_i$ and hence we get the desired result. 
\end{proof}

\begin{remark}\label{remark_ACgroup}
 It is easy to see that Proposition \ref{prop:divisibility_B_G(t)} is also true for the rational function $A_G(t)$. 
Now, suppose $G$ is an AC-group. Then in view of the proof of the Theorem \ref{thm:AC_AG=BG}, we  
have $B_G(t)=\frac{p(t)}{q(t)}$, where $p(t), q(t)\in \mathbb{Z}[x]$, and degree of $q(t)=(1-|Z(G)|t)\prod_{i=2}^r(1-n_it)$ is $r$.  The leading coefficient of $q(t)$  is equal to $(-1)^r|Z(G)|\prod_{i=2}^r n_i$.
\end{remark}

\section{Open problem}\label{question}
Theorem \ref{theorem:class-eq-A} asserts that two finite groups are $A$-equivalent if and only if they have the same class equation.
We do not know of an analogous result for $B$-equivalence.
\begin{question}
Is there a simple characterization of $B$-equivalence?
\end{question}

\section{Appendix}\label{GAP: code}
We record our GAP code used for computing $A_G(t)$ and $B_G(t)$.
\begin{Verbatim}%[breaklines=true, breakanywhere=true]

t:=Indeterminate(Rationals, "t");

B:=function(g)
if (IsAbelian(g)=true) then
return 1/(1-t*Size(g));
else
return (1+Sum( Filtered(ConjugacyClasses(g),x->Size(x)>1),
  i-> t*B(Centralizer(g,Representative(i))) ) )/(1-t*Size(Center(g)));
fi;
end;

A:=function(g)
return (Sum (Elements(g), x->1/(1-t*Size(Centralizer(g,x)))))/Size(g);
end;
\end{Verbatim}

\bibliographystyle{amsalpha}

\bibliography{references}

\providecommand{\bysame}{\leavevmode\hbox to3em{\hrulefill}\thinspace}
\providecommand{\MR}{\relax\ifhmode\unskip\space\fi MR }
\providecommand{\MRhref}[2]{%
  \href{http://www.ams.org/mathscinet-getitem?mr=#1}{#2}
}
\providecommand{\href}[2]{#2}
\begin{thebibliography}{HKR00}

\bibitem[Bio76]{JCB}
J.~C. Bioch, \emph{On {$n$}-isoclinic groups}, Indag. Math. \textbf{38} (1976),
  no.~5, 400--407.

\bibitem[GAP18]{GAP4}
The GAP~Group, \emph{{GAP -- Groups, Algorithms, and Programming, Version
  4.10.0}}, (2018).

\bibitem[Hal40]{Hall}
P.~Hall, \emph{The classification of prime-power groups}, J. Reine Angew. Math.
  \textbf{182} (1940), 130--141.

\bibitem[HKR00]{HKR}
Michael~J. Hopkins, Nicholas~J. Kuhn, and Douglas~C. Ravenel, \emph{Generalized
  group characters and complex oriented cohomology theories}, J. Amer. Math.
  Soc. \textbf{13} (2000), no.~3, 553--594.

\bibitem[KPP21]{KPP}
D.~Kaur, S.~K. Prajapati, and A.~Prasad, \emph{Simultaneous conjugacy classes
  of finite p-groups of rank $\leq$ 5}, Pre-print,
  http://arxiv.org/abs/2103.05234 (2021).

\bibitem[Roc75]{Rocke}
David~M. Rocke, \emph{p-groups with abelian centralizers}, Proceedings of the
  London Mathematical Society \textbf{s3-30} (1975), no.~1, 55--75.

\bibitem[Sha16]{MR3485060}
Uday~Bhaskar Sharma, \emph{Simultaneous similarity classes of commuting
  matrices over a finite field}, Linear Algebra Appl. \textbf{501} (2016),
  48--97.

\end{thebibliography}

\end{document}